\documentclass[11pt,a4paper]{amsart}
\usepackage[utf8]{inputenc}
\usepackage[english]{babel}
\usepackage{amsfonts,amssymb,amscd}
\usepackage{amsthm}
\usepackage{mathrsfs}
\usepackage{graphics}
\usepackage{hyperref}

\newtheorem{Theorem}{Theorem}

\newtheorem{Proposition}{Proposition}
\newtheorem{Corollary}{Corollary}

\newcommand{\Prob}{{\mathsf{P}}}
\newcommand{\Ha}{{\mathfrak{H}}}
\newcommand \SINE {{\mathscr S}}
\newcommand \Conf {{\mathrm {Conf}}}
\DeclareMathOperator{\im}{Im}

\begin{document}
\title[The speed of convergence in Soshnikov's Central Limit Theorem]{The speed of Convergence with respect\\to the Kolmogorov---Smirnov Metric\\ 
	in the Soshnikov Central Limit Theorem\\ for the Sine Process}

\author[A. I. Bufetov]{Alexander I. Bufetov}

\address{Steklov Mathematical Institute of the Russian Academy of Sciences,\newline\hspace*{\parindent}Moscow,\newline
\hspace*{\parindent}St. Petersburg State University, St. Petersburg,\newline
\hspace*{\parindent}Centre national de la recherche scientifique}

\email{bufetov@mi-ras.ru}

\date{}
%%\date{13 August 2024}
%\datedor{26 ноября 2023 г.}
%\dateprin{26 ноября 2023 г.}

%\doi{10.4213/faa9999}

%%\dedicatory{}

\begin{abstract}
	For  rescaled  additive functionals of the sine process, upper bounds are obtained for their speed of convergence to the Gaussian distribution with respect to the Kolmogorov---Smirnov metric. Under scaling with coefficient $R$ the Kolmogorov---Smirnov distance is bounded from above by $c/\log R$ for a smooth function and by $c/R$ for a function holomorphic in a horizontal strip.	
\end{abstract}

\thanks{
	The research was financially supported by
	the Ministry of Science and Higher Education of the Russian Federation in the framework of a scientific project under agreement No. 075-15-2024-631, and also by the grant 24-7-1-21 of the  "BASIS" Foundation for the Advancement of Mathematics and Theoretical Physics.
}

%\markboth{Speed of convergence in Soshnikov's Central Limit Theorem}{A.\,I.~Bufetov}

\maketitle

%%\begin{fulltext}

%%\subjclass{60F05}

%%\begin{keywords}
%%	sine process, Kolmogorov---Smirnov distance, Esseen inequality,\\ 
%%	Borodin---Okounkov---Jeronimo---Case formula, Hardy space in strip
%%\end{keywords}
%%% For English vesrion:
%\begin{keywords}
%слова, слова, слова, и все, как одно, ключевые
%\end{keywords}

Let $\Prob_\SINE$ be the sine process, the determinantal measure on the space $\Conf(\mathbb{R})$ of configurations on $\mathbb{R}$ corresponding to the sine kernel
\begin{equation*}
	\SINE(x,y)=\frac{\sin\pi(x-y)}{\pi(x-y)}.
\end{equation*}
For a Borel function $f\colon\mathbb{R}\to\mathbb{R}$ having compact support set
\begin{equation*}
	S_f(X)=\sum_{x\in X}f(x)
\end{equation*}
to be the additive functional on the space $\Conf(\mathbb{R})$, corresponding to the function~$f$. If $f\in L_1(\mathbb{R})$ then $S_f\in L_1(\Conf(\mathbb{R}),\Prob_\SINE)$ and
\begin{equation*}
	\mathbb{E}_{\Prob_\SINE}S_f=\int_{\mathbb{R}}f(x)\,dx.	
\end{equation*}
Denote $\overline{S}_f=S_f-\mathbb{E}_{\Prob_\SINE}S_f$ and recall that
\begin{equation}\label{eq:VarSf}
	\mathbb{E}|\overline{S}_f|^2=\frac{1}{2}\iint |f(x)-f(y)|^2\cdot |\SINE(x,y)|^2\,dxdy.
\end{equation}
The equality \eqref{eq:VarSf} allows one to extend the definition of $\overline{S}_f$ by continuity to all functions~$f$ such that the integral in the right-hand side of~\eqref{eq:VarSf} converges. Let $f(\,\cdot\,/R)$ stand for the function whose value at the point $t$ equals $f(t/R)$. Define the space $\dot H_{1/2}(\mathbb{R})$ as the completion of the space of smooth compactly supported functions on $\mathbb{R}$ with respect to the norm $\|f\|_{\dot H_{1/2}}$, defined by the formula
\begin{equation*}
	\|f\|_{\dot H_{1/2}}^2=\int_{\mathbb{R}}|\xi|\cdot |\widehat{f}(\xi)|^2\,d\xi.
\end{equation*}
For a function $f\in \dot H_{1/2}$ denote by the symbols $f_+,f_-$ the functions defined by the formulas $\widehat{f_+}=\widehat{f}\cdot\chi_{[0,+\infty)}$, $\widehat{f_-}=\widehat{f}\cdot\chi_{(-\infty,0)}$. Set also $\mathring{f}=f_--f_+$.

\begin{Theorem}\label{thm:Soshnikov}
	If $f\in\dot H_{1/2}(\mathbb{R})$ and $\mathring{f}\in L_\infty(\mathbb{R})$ then the family of additive functionals $\overline{S}_{f(\,\cdot\,/R)}$ converges in distribution to the Gaussian distribution $\mathcal{N}(0,\|f\|_{\dot H_{1/2}}^2/4\pi^2)$.
\end{Theorem}

Theorem~\ref{thm:Soshnikov} has been proved by A.~Soshnikov~\cite{Soshnikov} under somewhat more restrictive conditions on the function~$f$. Our next aim is to establish the speed of convergence in Theorem~\ref{thm:Soshnikov}.
Recall that the Kolmogorov---Smirnov distance $d_{KS}$ between two probability measures on $\mathbb{R}$ is defined as the supremum of the difference of their distribution functions. 
With a slight abuse of notation, we shall write, for example, $d_{KS}(Y,\mathcal{N}(0,1))$ to denote the Kolmogorov---Smirnov distance from the distribution of the random variable $Y$ to the standard Gaussian distribution.

Denote by the symbol $\dot H_1(\mathbb{R})$ the completion of the space compactly supported smooth functions with respect to the norm $\|f\|_{\dot H_1}$ defined by the formula
\begin{equation*}
	\|f\|_{\dot H_{1/2}}^2=\int_{\mathbb{R}}|\xi|^2\cdot |\widehat{f}(\xi)|^2\,d\xi=
	\int_{\mathbb{R}}|f'(t)|^2\,dt.
\end{equation*}

\begin{Theorem}[\cite{Bufetov-arXiv}]\label{thm:convrate-H1/2}
	Assume that $f\in\dot H_{1/2}(\mathbb{R})\cap \dot H_1(\mathbb{R})$ and $\mathring{f}\in L_\infty(\mathbb{R})$. Then there exists a constant $c(f)$ depending only on $\|f\|_{\dot H_{1/2}}$, $\|f\|_{\dot H_1}$, and $\|\mathring{f}\|_{L_\infty}$ such that for all $R>1$ we have
	\begin{equation*}
		d_{KS}\Bigl(\overline{S}_{f(\,\cdot\,/R)}, \mathcal{N}\bigl(0,\|f\|_{\dot H_{1/2}}^2/4\pi^2\bigr)\Bigr)\le \frac{c(f)}{\log R}.
	\end{equation*}
\end{Theorem}

The convergence is faster if the function~$\mathring{f}$ is holomorhic in a strip. For $\delta>0$ let $\mathcal{HL}(\delta)$ be the space of functions $f$ that are holomorphic and bounded in the strip $\{|\im z|<\delta\}$ and satisfy the inequality 
\begin{equation*}
	\sup_{|\delta_1|<\delta}\int_{\mathbb{R}}|f'(t+i\delta_1)|^2\,dt<+\infty.
\end{equation*}
Endow the space $\mathcal{HL}(\delta)$ with the norm
\begin{equation*}
	\|f\|_{\mathcal{HL}(\delta)}=\sup_{|\im z|<\delta}|f(z)|+
	\sup_{|\delta_1|<\delta}\biggl(\int_{\mathbb{R}}|f'(t+i\delta_1)|^2\,dt\biggr)^{1/2}.
\end{equation*} 

\begin{Theorem}\label{thm:convrate-holom}
	Let $\delta>0$. Then for any $f\in\dot H_{1/2}(\mathbb{R})$ such that $\mathring{f}\in\mathcal{HL}(\delta)$ there exists a constant $c(f)$, depending only on $\delta$, $\|\mathring{f}\|_{\mathcal{HL}(\delta)}$ and such that for all $R>1$ we have
	\begin{equation*}
		d_{KS}\Bigl(\overline{S}_{f(\,\cdot\,/R)}, \mathcal{N}\bigl(0,\|f\|_{\dot H_{1/2}}^2/4\pi^2\bigr)\Bigr)\le \frac{c(f)}{R}.
	\end{equation*}
\end{Theorem}

The key r\^ole in the proofs of Theorems~\ref{thm:convrate-H1/2}, \ref{thm:convrate-holom} is played by the scaling limit of the Borodin---Okoun\-kov---Jeronimo---Case formula \cite{BorodinOkounkov,GeronimoCase,BasorChen,Bufetov-FAA,Bufetov-arXiv}. Recall that the Hankel operator $\Ha(h)$ corresponding to the function $h$ is defined by the formula
\begin{equation*}
	\Ha(h)\varphi(s)=\frac{1}{2\pi}\int_0^{+\infty} \widehat{h}(s+t)\varphi(t)\,dt.	
\end{equation*}

Denote
\begin{equation*}
	V_f(\lambda)=\det\biggl(1-\chi_{(1,+\infty)}
	\Ha\biggl(\exp\biggl(\lambda\mathring{f}\biggl(\frac{\cdot}{2\pi}\biggr)\biggr)\biggr)
	\Ha\biggl(\exp\biggl(-\lambda\mathring{f}\biggl(\frac{\cdot}{2\pi}\biggr)\biggr)\biggr)
	\chi_{(1,+\infty)}\biggr).
\end{equation*}

\begin{Proposition}[\cite{Bufetov-FAA, Bufetov-arXiv}]
	If $f\in\dot{H}_{1/2}(\mathbb{R})$ is such that $\mathring{f}\in L_\infty(\mathbb{R})$, then
	\begin{equation*}
		\mathbb{E}_{\Prob_\SINE}\exp(\lambda\overline{S}_f)=
		\exp\biggl(\frac{\lambda^2\|f\|_{\dot H_{1/2}}^2}{4\pi^2}\biggr)\cdot V_f(\lambda).
	\end{equation*}
\end{Proposition}

We have the estimate
\begin{equation*}
	V_f(\lambda)\le
	\exp\biggl(
	\biggl\|\chi_{(1,+\infty)}
	\Ha\biggl(\exp\biggl(\lambda\mathring{f}\biggl(\frac{\cdot}{2\pi}\biggr)\biggr)\biggr)\biggr\|_{HS}\cdot
	\biggl\|\Ha\biggl(\exp\biggl(-\lambda\mathring{f}\biggl(\frac{\cdot}{2\pi}\biggr)\biggr)\biggr)
	\chi_{(1,+\infty)}\biggr\|_{HS}
	\biggr).
\end{equation*}

For $R\ge 1$ define the seminorm $\|h\|_{\mathcal{H}(R)}$  on the space $\dot H_{1/2}$ by the formula
\begin{equation*}
	\|h\|_{\mathcal{H}(R)}^2=\int_{1}^{+\infty}\xi |\widehat{h}(\xi+R)|^2\,d\xi.
\end{equation*}
We have
\begin{gather}\label{eq:norms-HS}
	\|\chi_{(1,+\infty)}\Ha(h)\|_{HS}=\|h\|_{\mathcal{H}(1)},\quad
	\|h\|_{\mathcal{H}(R)}=\|h(\,\cdot\,/R)\|_{\mathcal{H}(1)},\\[6pt]
	|V_{f(\,\cdot\,/R)}(\lambda)|\le \exp\biggl(
	\biggl\|\exp\biggl(\lambda \mathring{f}\biggl(\frac{\cdot}{2\pi}\biggr)\biggr)\biggr\|_{\mathcal{H}(R)}\cdot
	\biggl\|\exp\biggl(-\lambda \mathring{f}\biggl(\frac{\cdot}{2\pi}\biggr)\biggr)\biggr\|_{\mathcal{H}(R)}
	\biggr).\notag
\end{gather}
One can see from~\eqref{eq:norms-HS} that if
$\exp(\lambda\mathring{f})\in\dot H_{1/2}$, then
$\chi_{(1,+\infty)}\Ha(\exp(\lambda\mathring{f}))\in HS$. The relation
$\exp(\lambda\mathring{f})\in\dot H_{1/2}$ holds once $\mathring{f}\in\dot H_{1/2}\cap L_\infty$; one can also see that if $\mathring{f}\in\dot H_{1/2}\cap L_\infty$ then 
\begin{equation*}
\biggl\|\chi_{(1,+\infty)}\Ha\biggl(\exp\biggl(\lambda\mathring{f}\biggl(\frac{\cdot}{R}\biggr)\biggr)\biggr)\biggr\|_{HS}\to 0
\end{equation*}
as $R\to +\infty$. Theorem~\ref{thm:Soshnikov} is proven. We proceed to the proofs of Theorems~\ref{thm:convrate-H1/2} and~\ref{thm:convrate-holom}.
The Esseen smoothing inequality in the form of Feller~\cite{Feller,Petrov} yields

\begin{Corollary}\label{cor:dKS}
	Let $\varkappa_1,\varkappa_0, T\in\mathbb{R}_+$ and $X$ be a random variable such that its characteristic function $f_X(\xi)=\mathbb{E}\exp(i\xi X)$ has the form
	$f_X(\xi)=\exp(-\xi^2/2)W(\xi)$ with
	\begin{equation*}
		\sup_{|\xi|\le 1}|W'(\xi)|=\varkappa_1,\quad
		\sup_{|\xi|\le T}|W(\xi)-1|=\varkappa_0.
	\end{equation*}
	Then we have
	\begin{equation*}
		d_{KS}(X,\mathcal{N}(0,1))\le \varkappa_0+\varkappa_1+\frac{4}{T}.
	\end{equation*}
\end{Corollary}

The estimates on the speed of convergence are thus reduced to the estimates on the norms 
\begin{equation*}
	\biggl\|\chi_{(1,+\infty)}
	\Ha\biggl(\exp\biggl(\lambda\mathring{f}\biggl(\frac{\cdot}{2\pi}\biggr)\biggr)\biggr)\biggr\|_{HS},\quad
	\biggl\|\Ha\biggl(\exp\biggl(-\lambda\mathring{f}\biggl(\frac{\cdot}{2\pi}\biggr)\biggr)\biggr)
	\chi_{(1,+\infty)}\biggr\|_{HS}.
\end{equation*}

\begin{proof}[Proof of Theorem~\ref{thm:convrate-H1/2}]
	On the intersection $\dot H_1\cap L_\infty$ consider the norm $\|f\|_{\dot H_1\cap L_\infty}=\|f\|_{\dot H_1}+\|f\|_{L_\infty}$.
	If $f\in \dot H_1\cap L_\infty$ then $\exp(f)\in\dot H_1\cap L_\infty$, and we have
	\begin{equation*}
		\|\exp(f)\|_{\dot H_1\cap L_\infty}\le (1+\|f\|_{\dot H_1\cap L_\infty})\cdot
		\exp(\|f\|_{\dot H_1\cap L_\infty}).
	\end{equation*}
	Definitions directly imply that $\|f\|_{\mathcal{H}(R)}\le \|f\|_{\dot H_1}/\sqrt{R}$, whence we also have 
	\begin{equation*}
		\|\exp(\lambda \mathring{f}(\,\cdot\,/R))\|_{\mathcal{H}}\le
		\frac{(1+|\lambda|\cdot \|\mathring{f}\|_{\dot H_1\cap L_\infty})\exp(\lambda \|\mathring{f}\|_{\dot H_1\cap L_\infty})}{\sqrt{R}}.
	\end{equation*} 
	The derivative $(d/d\lambda)V_f(\lambda)$ for  $\lambda\in[-1,1]$ is estimated by the Cauchy formula
	\begin{equation*}
		\frac{d}{d\lambda}V_f(\lambda)=\frac{1}{2\pi i}\oint_{|\xi-\lambda|=1}\frac{V_f(\xi)}{(\xi-\lambda)^2}\,d\xi.
	\end{equation*}
	Now Corollary~\ref{cor:dKS} implies Theorem~\ref{thm:convrate-H1/2}.
\end{proof}

\begin{proof}[Proof of Theorem~\ref{thm:convrate-holom}]
	From the definition of the space $\mathcal{HL}(\delta)$ one directly obtains that
	\begin{enumerate}
		\item if $f\in\mathcal{HL}(\delta)$ then $\exp(f)\in\mathcal{HL}(\delta)$, and
		\begin{equation*}
			\|\exp(f)\|_{\mathcal{HL}(\delta)}\le \exp(\|f\|_{\mathcal{HL}(\delta)})\cdot (1+\|f\|_{\mathcal{HL}(\delta)});
		\end{equation*}
		\item for any $\delta_1<\delta$ and any $f\in\mathcal{HL}(\delta)$ we have
		\begin{equation*}
			\|f\|_{\mathcal{H}(R)}\le \exp(-\delta_1R)\|f\|_{\mathcal{HL}(\delta)};		
		\end{equation*}
		\item if $\mathring{f}\in\mathcal{HL}(\delta)$, then for any $\delta_1<\delta$ there exists a constant $A(\mathring{f})$, depending only on $\delta$, $\|\mathring{f}\|_{\mathcal{HL}(\delta)}$ such that
		\begin{equation*}
			\|\exp(\lambda\mathring{f}(\,\cdot\,/R))\|_{\mathcal{H}}\le\exp(A|\lambda|-\delta_1R).
		\end{equation*}
	\end{enumerate}
	The derivative $(d/d\lambda)V_f(\lambda)$ for $\lambda\in[-1,1]$ is again estimated by the Cauchy formula. Corollary~\ref{cor:dKS} now implies Theorem~\ref{thm:convrate-holom}. 
\end{proof}

For the determinantal process with Bessel kernel the speed of convergence in the Central Limit Theorem was recently estimated by S.~Gorbunov~\cite{Gorbunov}.

%%\end{fulltext}

\enlargethispage{2\baselineskip}

\end{document}